\documentclass[a4paper, 12pt]{amsart}

\usepackage[a4paper]{geometry}
\usepackage{amssymb}
\usepackage{amsthm}
\usepackage{amsmath}
\usepackage{graphicx}
\usepackage{hyperref}
\usepackage{multirow}
\usepackage{array}
\usepackage[colorinlistoftodos]{todonotes}
\usepackage[english]{babel}
\usepackage{lmodern}
\usepackage{comment}
\usepackage{bbold}
\usepackage{mathtools}
\usepackage{enumerate}
\usepackage{cite}

\newcommand{\genlegendre}[4]{%
	\genfrac{(}{)}{}{#1}{#3}{#4}%
	\if\relax\detokenize{#2}\relax\else_{\!#2}\fi
}
\newcommand{\legendre}[3][]{\genlegendre{}{#1}{#2}{#3}}

\allowdisplaybreaks

\usepackage[utf8]{inputenc}
\usepackage[T1]{fontenc}

\newtheorem{theorem}{Theorem}[section]
\newtheorem{lemma}[theorem]{Lemma}
\newtheorem{corollary}[theorem]{Corollary}
\newtheorem*{conjecture}{Conjecture}

\theoremstyle{definition}
\newtheorem{definition}[theorem]{Definition}
\newtheorem{proposition}[theorem]{Proposition}

\newtheorem*{question}{Question}
\newtheorem{example}[theorem]{Example}

\theoremstyle{remark}
\newtheorem{remark}[theorem]{Remark}
\newcommand{\longcomment}[1]{}

\DeclareMathOperator{\Frob}{Frob}

\DeclareMathOperator{\rank}{rank}

\DeclareMathOperator{\Gal}{Gal}

\DeclareMathOperator{\OO}{\mathcal{O}}
\DeclareMathOperator{\disc}{\operatorname{disc}}
\DeclarePairedDelimiter\abs{\lvert}{\rvert}

\newcommand{\Mod}[1]{\ (\mathrm{mod}\ #1)}
\newcommand{\bigO}[1]{O\left(#1\right)}
\newcommand{\smallO}[1]{o\left(#1\right)}

\newcommand{\whichbold}[1]{\mathbb{#1}} 

\newcommand{\Z}{\whichbold{Z}}
\newcommand{\FF}{\whichbold{F}}
\newcommand{\R}{\whichbold{R}}

\newcommand{\Q}{\whichbold{Q}}

\newcommand{\Fp}{\whichbold{F}_{p}}
\newcommand{\C}{\whichbold{C}}
\newcommand{\Qp}{\whichbold{Q}_{p}}
\newcommand{\N}{\whichbold{N}}


\numberwithin{equation}{section}
\author{Lukas Novak}
\address{Department of Mathematics\\ 
	University of Zagreb\\
	 Bijeni\v{c}ka cesta 30\\
	  10000 Zagreb\\
	  Croatia}
\email{lukas.novak@math.hr}

\title[Quadratic twists of genus one curves]{Quadratic twists of genus one curves}

\makeatletter
\renewcommand{\tocsection}[3]{%
	\indentlabel{\@ifnotempty{#2}{\bfseries\ignorespaces#1 #2\quad}}\bfseries#3}
\renewcommand{\tocsubsection}[3]{%
	\indentlabel{\@ifnotempty{#2}{\ignorespaces#1 #2\quad}}#3}

\newcommand\@dotsep{4.5}
\def\@tocline#1#2#3#4#5#6#7{\relax
	\ifnum #1>\c@tocdepth 
	\else
	\par \addpenalty\@secpenalty\addvspace{#2}%
	\begingroup \hyphenpenalty\@M
	\@ifempty{#4}{%
		\@tempdima\csname r@tocindent\number#1\endcsname\relax
	}{%
		\@tempdima#4\relax
	}%
	\parindent\z@ \leftskip#3\relax \advance\leftskip\@tempdima\relax
	\rightskip\@pnumwidth plus1em \parfillskip-\@pnumwidth
	#5\leavevmode\hskip-\@tempdima{#6}\nobreak
	\leaders\hbox{$\m@th\mkern \@dotsep mu\hbox{.}\mkern \@dotsep mu$}\hfill
	\nobreak
	\hbox to\@pnumwidth{\@tocpagenum{\ifnum#1=1\bfseries\fi#7}}\par
	\nobreak
	\endgroup
	\fi}
\AtBeginDocument{%
	\expandafter\renewcommand\csname r@tocindent0\endcsname{0pt}
}
\def\l@subsection{\@tocline{2}{0pt}{2.5pc}{5pc}{}}
\makeatother

\begin{document}
\maketitle

\begin{abstract}
    For a given irreducible and monic polynomial $f(x) \in \Z[x]$ of degree $4$, we consider the quadratic twists by square-free integers $q$ of the genus one quartic ${H\, :\, y^2=f(x)}$ 
    \[
        H_q \, :\, qy^2=f(x).
    \]
    
    Let $L$ denotes the set of positive square-free integers $q$ for which $H_q$ is everywhere locally solvable. For a real number $x$, let ${L(x)= \#\{q\in L:\, q \leq x\}}$ be the number of elements in $L$ that are less then or equal to $x$.

    In this paper, we obtain that
    \[
        L(x) = c_f \frac{x}{(\ln{x})^{m}}+\bigO{\frac{x}{(\ln{x})^\alpha}}
    \]
    for some constants $c_f>0$, $m$ and $\alpha$ only depending on $f$ such that $m<\alpha \leq 1+m$.
    We also express the Dirichlet series $F(s)=\sum_{n \in L} n^{-s}$ associated to the set $L$ in terms of Dedekind zeta functions of certain number fields. 
\end{abstract}

\tableofcontents

\section{Introduction}\label{sec:1}

Currently, a large amount of research is being done on studying elliptic curves and their twists. Many new discoveries and progress have been made in this area. However, despite the large amount of progress there are still a lot of open questions. 

One such question, concerning quadratic twists of elliptic curve, is the Goldfeld's conjecture~\cite{Goldfeld1979}.

\begin{conjecture}[\textbf{Goldfeld's conjecture}]
    Let $E$ be an elliptic curve over $\Q$. Denote with $\rank{(E)}$ the rank of the elliptic curve $E$, with $E^d$ the quadratic twist of $E$ by a square-free integer $d$ and with
    \[
        S(x)=\{\text{square-free } d\in \Z:\, \abs{d}\leq x\}
    \]
    the set of all square-free integers between $-x$ and $x$, for a real number $x$. Then the following holds
    \[
        \lim_{x\to \infty}\frac{\sum_{d\in S(x)}\rank{(E^d)}}{\#S(x)}=\frac{1}{2}.
    \]
\end{conjecture}
Smith, in~\cite{smith20172inftyselmer} (see Corollary $1.3$), has proven (assuming the Birch and Swinnerton-Dyer conjecture for a certain set of elliptic curves) that the Goldfeld's conjecture is true for elliptic curves $E/\Q$ that have full rational $2$-torsion and don't have a rational cyclic subgroup of order $4$.

The above conjecture has also a very interesting consequence about the rank of quadratic twists of elliptic curves. By assuming that the Goldfeld's conjecture and the parity conjecture are true it follows that for $100\%$ of square-free integers $d$ the rank of $E^d/\Q$ is ether $0$ or $1$.  

Another interesting thing to study about quadratic twists of elliptic curve are their $2$-Selmer groups. A lot of progress is made in studying the dimension and the distribution of $2$-Selmer groups of the quadratic twists of elliptic curve and following articles only illustrate some of the results.

Barrera Salazar, Pacetti and Tornaría in~\cite{MR4477669} have obtained a lower and an upper bound for the $2$-Selmer rank of an elliptic curve $E$ over a number field $K$ without a $K$-rational point of order $2$ (under certain hypotheses on the reduction of $E$ modulo primes of $K$). As an application, they proved that under certain mild hypotheses a positive portion of the quadratic twists of $E$ have the same $2$-Selmer group.

Morgan and Paterson in~\cite{MR4400944} studied the distribution of the group $\operatorname{Sel}^2(E^d/K)$, for a fixed elliptic curve $E/\Q$ and a quadratic extension $K/\Q$, as $d$ varies over square-free integers. They have shown that, when considering square-free integers $\abs{d}<X$, the random variable
\[
    \frac{\operatorname{dim}_{\FF_2}\operatorname{Sel}^2(E^d/K)-\log{\log{\abs{d}}}}{\sqrt{2\log{\log{\abs{d}}}}}
\]
converges to the standard normal distribution as $X \to \infty$. It follows that for any fixed real number $z$, $0\%$ of quadratic twists satisfy $\operatorname{dim}_{\FF_2}\operatorname{Sel}^2(E^d/K) \leq z$. In particular, the average size of $\operatorname{Sel}^2(E^d/K)$ is infinite. Additionally, if $E$ has no rational cyclic $4$-isogeny, then the average size of $\operatorname{Sel}^2(E^d/\Q)$ is finite and the probability that $\operatorname{dim}_{\FF_2}\operatorname{Sel}^2(E^d/\Q)=n$ is positive for every integer $n\geq 2$.

Yu in~\cite{MR3747176} has proven that for an arbitrary elliptic curve $E$ over an arbitrary number field $K$, if the set $A_E$ of $2$-Selmer ranks of quadratic twists of $E$ contains an integer $c$, then it contains all integers larger than $c$ having the same parity as $c$. In addition, Yu also finds sufficient conditions on the set $A_E$ such that $A_E$ is equal to $\Z_{\geq t_E}$ for some number $t_E$ and gives an upper bound for $t_E$ in case all points in $E[2]$ are rational.

Instead of elliptic curves we can also consider quartics of genus one. One interesting question for genus one quartics over $\Q$ is whether they have a rational point. Namely, if a genus one quartic $C/\Q$ has a rational point then the curve $C$ is birationally isomorphic to an elliptic curve $E$. In this case studying rational points on the curve $C$ reduces to studying rational points on the elliptic curve $E$.

Surprisingly, for genus one quartics not much is known about the existence of a rational point. The following articles deal with the question whether quadratic twists of a genus one quartic have a rational point or not.

Kazalicki in his recent work~\cite{kazalicki2022quadratic} connects rational Diophantine $D(d)$-quintuples with quadratic twists of a genus one quartic $H$ and their $2$-Selmer groups. In his work he also partially answers the question for which primes $p$ the quadratic twist of $H$ by $p$ has a rational point.

\c Ciperiani and Ozman in \cite{ciperiani2015local} have found necessary and sufficient conditions for a point $P$ of an elliptic curve to lie in the image of the global trace map. These conditions can be used to determine whether a quadratic twist of a genus one curve has a rational points.

Motivated by those works, we are inspired to study the quadratic twists of genus one quartics. A natural first step in answering the question whether such twists have a rational point is examining the local properties of those twists. Thus, we try to answer the following question.

Let $f(x) \in \Z[x]$ be a irreducible and monic polynomial of degree $4$ and ${H\, :\, y^2=f(x)}$ the corresponding quartic of genus one. For a square-free integer $q$ we consider the quadratic twist of the curve $H$
\[
    H_q \, :\, qy^2=f(x).
 \]
Denote with $L=\{q\in \N :\, q \text{ is square-free and } H_q \text{ is ELS}\}$ the set of positive square-free integers $q$ for which $H_q$ is \textbf{everywhere locally solvable} (ELS), i.e.\ has a solution in $\R$ and in $\Qp$ for every prime $p$. For a real number $x$, let $L(x)= \#\{q\in L:\, q\leq x\}$ be the number of elements in $L$ that are less then or equal to $x$.
\begin{question}
    What is the asymptotic behaviour of the function $L(x)$ as $x \to \infty$?
\end{question}

We recall the following definitions for the factorization type of a polynomial and of an ideal.

\begin{definition}\label{def_fact_type:polynom}
    Let $f(x) \in \Z[x]$ be a polynomial with integer coefficients and $p$ a prime number.
    If $f(x)\Mod{p} = f_1(x)^{e_1}\dotsm f_g(x)^{e_g}$ where polynomials $f_1$, \ldots, $f_q \in \Z/p\Z[x]$ are irreducible, then we say that $f(x)$ has \textbf{factorization type} $(\operatorname{deg}(f_1),\, \operatorname{deg}(f_2),\, \dotsc ,\, \operatorname{deg}(f_g))$ modulo $p$.
\end{definition} 

\begin{definition}\label{def_fact_type:ideal}
    Let $K$ be a number field and $p$ a (rational) prime number. If $p\OO_K = P_1^{e_1}\dotsm P_g(x)^{e_g}$ where $P_1$, \ldots, $P_g$ are prime ideals in $K$, then we say that the ideal $(p)$ in the ring of integers $\OO_K$ has \textbf{factorization type} $(f_1,\, f_2,\, \dotsc,\, f_g)$ where $f_i$ is the inertia degree of $P_i$ above $p$.
\end{definition} 

\begin{remark}
    For the sake of brevity, if the ideal $(p)$ in the ring of integers $\OO_K$ has factorization type $(f_1,\, f_2,\, \dotsc,\, f_g)$, we instead say that the prime $p$ has factorization type $(f_1,\, f_2,\, \dotsc,\, f_g)$. 
\end{remark}

We first start with exploring the necessary and sufficient conditions on positive square-free integers $q$ for which $H_q$ is ELS. The obtained criterion tells us that $H_q$ is ELS if and only if all prime factors $p$ of $q$, that don't divide $\disc(f)$, have factorization type $(1,1,1,1)$, $(1,1,2)$ or $(1,3)$ and $q$ satisfies certain congruence conditions modulo $8$ and modulo odd prime factors of $\disc(f)$ (see Proposition~\ref{els_crit} for more details).

Next, we define the \textbf{generating Dirichlet series} $F(s)$ corresponding to the set $L$ (i.e.\ the set of positive square-free integers $q$ for which $H_q$ is ELS) as
\[
    F(s)=\sum_{n \in L} \frac{1}{n^s}.
\]

We also define an \textbf{auxiliary Dirichlet series} for $F(s)$
\[
    g(s)=\sideset{}{'}\prod\left(1+p^{-s}\right)
\]
where $\sideset{}{'}\prod$ denotes the product over primes $p$ that have factorization type $(1,1,1,1)$, $(1,1,2)$ or $(1,3)$. We are interested in these primes because they appear as prime factors in our square-free integers $q$ for which $H_q$ is ELS due to Proposition~\ref{els_crit}.

For a Dirichlet series $D(s)=\sum_{n=1}^\infty \frac{a_n}{n^s}$ and a Dirichlet character $\chi$ we denote with\newline ${D^\chi(s) = \sum_{n=1}^\infty \dfrac{a_n \cdot \chi(n)}{n^s}}$ \textbf{the twist of} $D$ \textbf{with the character} $\chi$.

Let $\chi_1$, $\chi_2$, $\chi_3$ and $\chi_4$ be all Dirichlet characters modulo $8$ defined as in the table:
\begin{center}
    \begin{tabular}{| c | c | c | c | c | }
    \hline
    $n$ & $1$ & $3$ & $5$ & $7$ \\
    \hline
    $\chi_1(n)$ & $1$ & $1$ & $1$ & $1$ \\
    \hline
    $\chi_2(n)$ & $1$ & $1$ & $-1$ & $-1$ \\
    \hline
    $\chi_3(n)$ & $1$ & $-1$ & $-1$ & $1$ \\
    \hline
    $\chi_4(n)$ & $1$ & $-1$ & $1$ & $-1$ \\
    \hline
    \end{tabular}
\end{center}
and let $\psi_r$ be the quadratic Dirichlet character modulo $r$, i.e.\ ${\psi_r (n) = \legendre{n}{r}}$.

Using the previous criterion we are able to express the generating Dirichlet series $F(s)$ via the auxiliary Dirichlet series $g(s)$ and its twists by the above defined Dirichlet characters (see Theorem~\ref{thm_F(s)_formula} for more details). To better illustrate the result in Theorem~\ref{thm_F(s)_formula} we give the following example:

\begin{example}\label{example_2}
    Let $f(x)=x^4-x+1$ be the same polynomial as in Example~\ref{example_1}. In Example~\ref{example_1} we described all the congruence conditions that $q$ has to satisfy modulo $8$ and modulo $229$ with sets $C_1$, $C_2$, \ldots , $C_{12}$. By plugging all this into Theorem~\ref{thm_F(s)_formula} and simplifying the expression we get
    \[
        F(s)=g(s)+\frac{1}{2}\cdot\frac{1}{229^s}g(s)+\frac{1}{2}\cdot\frac{1}{229^s}g^{\psi_{229}}(s).
    \]
\end{example}

Observe that the Dirichlet series $g(s)$ and its twists by Dirichlet characters correspond to \textbf{multiplicative frobenian functions} (see Definition~\ref{def_2.1.} for more details). This allows us to use the Landau-Selberg-Delange method developed for frobenian functions described in Section~\ref{sec:frobenian_functions} (following \cite{loughran2022frobenian}). By using this method we were finally able to calculate the asymptotic behaviour of $L(x)$.

\begin{corollary}\label{L(x)_formula_shortv}
    Let $G$ be the Galois group of the polynomial $f(x)$.  With the notation as above we have that
    \[
        L(x) = c_f \frac{x}{(\ln{x})^{m}}+\bigO{\frac{x}{(\ln{x})^\alpha}}
    \]
    where $c_f>0$ is a constant depending only on the polynomial $f(x)$,
    \[
    m= 
    \begin{cases}
        3/4, & \text{if } G \cong \left(\Z/2\Z\right) \times \left(\Z/2\Z\right) \text{ or } G \cong \Z/4\Z \\
        5/8, & \text{if } G \cong D_4 \\
        1/4, & \text{if } G \cong A_4 \\
        3/8, & \text{if } G \cong S_4\\
    \end{cases}
    \]
    and $\alpha$ is a constant depending only on the polynomial $f(x)$ such that ${m<\alpha \leq 1+m}$.
\end{corollary}

Lastly, motivated by Novak's Master thesis~\cite{novak2022lokalna}, in Proposition~\ref{thm_g(s)_via_dedekind} we give formulas for the auxiliary Dirichlet series $g(s)$ via Dedekind zeta functions of certain number fields. 

With this formulas for $g(s)$ one might try to apply Perron's formula and then estimate the integral using the "key-hole" contour as described in \cite{novak2022lokalna} (Chapter $4$) or another contour as described in \cite{serre1974divisibilite} (Chapter~$2$) or try to apply the Selberg-Delange method for Dirichlet series sufficiently close to a (complex) power of $\zeta (s)$ to evaluate the summatory function of the coefficients as in \cite{tenenbaum2015introduction} (Chapter~$2.5$) in order to obtain better estimates for the error term in Corollary~\ref{L(x)_formula_shortv}.

\begin{proposition}\label{thm_g(s)_via_dedekind}
    Let $g(s)$ be the auxiliary Dirichlet series defined as before, $G$ the Galois group of the polynomial $f(x)$, $L$ the splitting field of the polynomial $f(x)$ and $K=\Q(\alpha)$ where $\alpha \in \C$ is a root of the polynomial $f(x)$. The following relations hold for $g(s)$:
    \begin{enumerate}[i)]
        \item If $G \cong \left(\Z/2\Z\right) \times \left(\Z/2\Z\right)$ or $G \cong \Z/4\Z$, then
        \[  
            g(s)^4 \sim \frac{\zeta_K (s)}{\zeta_K (2s)}.
        \]
        \item If $G \cong D_4$, then
        \[
            g(s)^8 \sim \left( \frac{\zeta_K (s)}{\zeta_K (2s)} \right)^4 \cdot \left( \frac{\zeta_L (s)}{\zeta_L (2s)} \right)^{-1}.
        \]
        \item If $G \cong A_4$, then
        \[
            g(s)^4 \sim \left( \frac{\zeta_K (s)}{\zeta_K (2s)} \right)^4 \cdot \left( \frac{\zeta_L (s)}{\zeta_L (2s)} \right)^{-1}.
        \]
        \item If $G \cong S_4$, then
        \[
            g(s)^{24} \sim \left( \frac{\zeta_K (s)}{\zeta_K (2s)} \right)^{12} \cdot \left( \frac{\zeta_L (s)}{\zeta_L (2s)} \right)^{-5} \cdot \left( \frac{\zeta_{L^{\Z/3\Z}} (s)}{\zeta_{L^{\Z/3\Z}} (2s)} \right)^6.
        \]
    \end{enumerate}
    Here $g_1 (s) \sim g_2 (s)$ means that $g_1 (s)/g_2 (s)$ is a "nice" function (holomorphic on $\operatorname{Re}(s)>\frac{1}{2}$ and bounded on the strips $\operatorname{Re}(s)=\frac{1}{2}+\delta$ for every $\delta >0$) and $L^{\Z/3\Z}$ is the fixed field of $L$ by a subgroup $H \cong \Z/3\Z$ of the Galois group $G$.
\end{proposition}

\section{Proof of the ELS criterion}\label{sec:2}

In this section we prove the following criterion for positive square-free $q$'s for which $H_q$ is everywhere locally solvable.

\begin{proposition}[ELS criterion] \label{els_crit}
    Let $f(x) \in \Z[x]$ be a irreducible and monic polynomial of degree $4$, $K=\Q(\alpha)$ a number field where $\alpha$ is a root of $f(x)$. Let $q$ be a positive square-free integer and ${H_q \, :\, qy^2=f(x)}$ a curve. The curve $H_q$ is ELS if and only if the following conditions hold for $q$:
    \begin{enumerate}[i)]
        \item If a prime $p \nmid \disc(f)$ and $p \mid q$, then the ideal $(p)$ in the ring of integers $\OO_K$ has a factorization type $(1,1,1,1)$, $(1,1,2)$ or $(1,3)$.
        \item If a prime $p \mid \disc(f)$ and $p>2$. We choose any integer $u$ such that $\legendre{u}{p}=-1$ and observe the curves $H_1$, $H_u$, $H_p$ and $H_{up}$ and their solubility in $\Qp$. 
        
        For $p \nmid q$ we look at the following cases:
        \begin{enumerate}[1)]
            \item If $H_u$ has a solution in $\Qp$ then there are no conditions for $q$ modulo $p$.
            \item If $H_u$ doesn't have a solution in $\Qp$ than $\legendre{q}{p}=1$.
        \end{enumerate}
        
        For $p\mid q$ we write $q=pq_1$ and look at the following cases:
        \begin{enumerate}[1)]
            \item If $H_p$ and $H_{up}$ have a solution in $\Qp$, then there are no conditions for $q$ modulo $p$.
            \item If $H_p$ does and $H_{up}$ doesn't have a solution in $\Qp$, then $\legendre{q_1}{p}=1$.
            \item If $H_p$ doesn't and $H_{up}$ does have a solution in $\Qp$, then $\legendre{q_1}{p}=-1$.
            \item If $H_p$ and $H_{up}$ haven't a solution in $\Qp$, then the prime $p \nmid q$.
        \end{enumerate}
        \item If a prime $p \mid \disc(f)$ and $p=2$. We look at $\{1,-1,5,-5,2,-2,10,-10\}$ the set of representatives for $\Q_2^{\times}/(\Q_2^{\times})^2$ and the curves $H_1$, $H_{-1}$, $H_{5}$, $H_{-5}$, $H_{2}$, $H_{-2}$, $H_{10}$ and $H_{-10}$.
        
        For $2 \nmid q$ the conditions for $q$ modulo $8$ are the union of the following conditions: 
        
        If $H_u$ has a solution in $\Q_2$, then $q \equiv u \Mod{8}$, for $u \in \{\pm 1, \pm 5\}$.

        For $2 \mid q$ we write $q=2q_1$ and the conditions for $q_1$ modulo $8$ are the union of the following conditions : 
        
        If $H_{2u}$ has a solution in $\Q_2$, then $q_1 \equiv u \Mod{8}$, for $u \in \{\pm 1, \pm 5\}$.
    \end{enumerate}
\end{proposition}

In the following example we demonstrate which conditions modulo $8$ and modulo prime factors of $\disc(f)$ we get after applying this criterion.

\begin{example}\label{example_1}
    Let $f(x)=x^4-x+1$. We have that $\disc(f)=229$ and let us choose $u=2$ as our non-square residue modulo $229$. 
    
   One can check (for example using MAGMA) that the curves $H_1$, $H_2$ and $H_{229}$ have a point in $\Q_{229}$ while the curve $H_{2\cdot 229}$ doesn't. Similarly, we check that curves $H_1$, $H_{-1}$, $H_{5}$ and $H_{-5}$ have a point in $\Q_2$ while the curves $H_2$, $H_{-2}$, $H_{10}$ and $H_{-10}$ don't have a point in $\Q_2$.

    Let us denote the conditions modulo $8$ and modulo $229$ that we get from the above criterion with the following sets:

    \begin{align*}
        C_1 &= \left\{ q\in\N \text{ square-free }:\, \legendre{q}{229}=1,\, q\equiv 1\Mod{8} \right\}, \\
        C_2 &= \left\{ q\in\N \text{ square-free } :\, \legendre{q}{229}=1,\, q\equiv 3\Mod{8} \right\}, \\
        C_3 &= \left\{ q\in\N  \text{ square-free }:\, \legendre{q}{229}=1,\, q\equiv 5\Mod{8} \right\}, \\
        C_4 &= \left\{ q\in\N \text{ square-free } :\, \legendre{q}{229}=1,\, q\equiv 7\Mod{8} \right\}, \\
        C_5 &= \left\{ q\in\N \text{ square-free } :\, \legendre{q}{229}=-1,\, q\equiv 1\Mod{8} \right\}, \\
        C_6 &= \left\{ q\in\N \text{ square-free } :\, \legendre{q}{229}=-1,\, q\equiv 3\Mod{8} \right\}, \\
        C_7 &= \left\{ q\in\N \text{ square-free } :\, \legendre{q}{229}=-1,\, q\equiv 5\Mod{8} \right\}, \\
        C_8 &= \left\{ q\in\N \text{ square-free } :\, \legendre{q}{229}=-1,\, q\equiv 7\Mod{8} \right\}, \\
        C_9 &= \left\{ 229q_1 \text{ square-free } :\,q_1\in\N,\, \legendre{q_1}{229}=1,\, q_1\equiv 1\Mod{8} \right\}, \\
        C_{10} &= \left\{ 229q_1 \text{ square-free } :\,q_1\in\N,\, \legendre{q_1}{229}=1,\, q_1\equiv 3\Mod{8} \right\}, \\
        C_{11} &= \left\{ 229q_1 \text{ square-free } :\,q_1\in\N,\, \legendre{q_1}{229}=1,\, q_1\equiv 5\Mod{8} \right\}, \\
        C_{12} &= \left\{ 229q_1 \text{ square-free } :\,q_1\in\N,\, \legendre{q_1}{229}=1,\, q_1\equiv 7\Mod{8} \right\}.
    \end{align*}

    From Proposition~\ref{els_crit} we have that all positive square-free integers $q$ for which $H_q$ is ELS are exactly does in the set $\cup_{i=1}^{12}C_i$ and whose prime factors, except eventually $229$, have factorisation type $(1,1,1,1)$, $(1,1,2)$ or $(1,3)$. 
\end{example}

\begin{proof}[Proof of Proposition~\ref{els_crit}]
Let's suposs that $H_q$ is ELS.
\begin{enumerate}[i)]
    \item Let $p$ be a prime such that $p \nmid \cdot \disc(f)$ and $p \mid q$. Since $H_q$ has a solution in $\Qp$ it also has a solution in $\Fp$. Using the fact that $p \mid q$ we deduce that the polynomial $f(x)$ has a root modulo $p$. Therefore $f(x)$ has a factorization type $(1,1,1,1)$, $(1,1,2)$ or $(1,3)$ modulo $p$. By Dedekind's theorem it follows now that the ideal $(p)$ in $\OO_K$ has the same factorization type as $f(x)$ modulo $p$, that is $(1,1,1,1)$, $(1,1,2)$ or $(1,3)$.

    \item Let $p>2$ be a prime such that $p \mid \disc(f)$. We will use the following lemma:
    \begin{lemma}\label{lemma_iso}
        Let $q$ and $r$ be integers and $p$ a prime. If $q \sim r$ in $\Qp^\times/(\Qp^\times)^2$ (i.e.\ $q=a^2r$ for some $a\in \Qp$), then $H_q(\Qp)$ and $H_r(\Qp)$ are isomorphic.
    \end{lemma}
    \begin{proof}[Proof of the Lemma]
        Since $q \sim r$ there exists $a \in \Qp^\times$ such that $q=a^2r$. The isomorphism ${\varphi : E_r(\Qp) \to H_q(\Qp)}$ is then defined with $\varphi(x,y)=(x,ay)$. 
    \end{proof}
    Now observe that for square-free integers $q$ the following holds:
    \begin{enumerate}
        \item $q \sim 1$ in $\Qp^\times/(\Qp^\times)^2$ if and only if $\legendre{q}{p}=1$,

        \item $q \sim u$ in $\Qp^\times/(\Qp^\times)^2$ if and only if $\legendre{q}{p}=-1$,

        \item $q \sim p$ in $\Qp^\times/(\Qp^\times)^2$ if and only if $q=pq_1$ and $\legendre{q_1}{p}=1$,

        \item $q \sim pu$ in $\Qp^\times/(\Qp^\times)^2$ if and only if $q=pq_1$ and $\legendre{q_1}{p}=-1$.
    \end{enumerate} 
    Now depending whether $q \sim 1$, $u$, $p$ or $pu$ in $\Qp^\times/(\Qp^\times)^2$ using the above equivalences and Lemma~\ref{lemma_iso} we get the desired conditions on $q$.  

    \item Integers that are squares in $\Q_2$ are of the form $2^{2k}l$ where $k$ is a non-negative integer and $l \equiv 1 \Mod{8}$. Using this we get that for square-free integers $q$ the following holds:
    \begin{enumerate}
        \item $q \sim \pm 1$, $\pm 5$ in $\Q_2^\times/(\Q_2^\times)^2$ if and only if $q \equiv \pm 1$, $\pm 5 \Mod{8}$ (respectively),

        \item $q \sim \pm 2$, $\pm 10$ in $\Q_2^\times/(\Q_2^\times)^2$ if and only if $q=2q_1$ and $q_1 \equiv \pm 1$, $\pm 5 \Mod{8}$ (respectively).
    \end{enumerate}
    Again by using the above equivalences and Lemma~\ref{lemma_iso} we get the desired conditions.
\end{enumerate}

Now let us assume that the conditions i) - iii) in Theorem~\ref{els_crit} hold. $H_q$ has obviously a solution in $\R$ since $q>0$ and the polynomial $f(x)$ is monic. It remains to show that $H_q$ has a solution in $\Qp$ for every prime $p$. Let $p$ be a prime.
\begin{enumerate}[$1^\circ$]
    \item If $p \nmid 2q\disc(f)$, then we have that $H_q$ is a genus $1$ curve and has good reduction modulo $p$. By Theorem $3$ in \cite{siksek2007genus} we immediately get that $H_q$ has a solution in $\Qp$.

    \item If $p \mid q$ and $p \nmid \disc(f)$, then we have that $(p)$ has factorization type $(1,1,1,1)$, $(1,1,2)$ or $(1,3)$ in $\OO_K$. By Dedekind's theorem it follows that the polynomial $f(x)$ has also factorization type $(1,1,1,1)$, $(1,1,2)$ or $(1,3)$ modulo $p$. This means that $f(x)$ has a root $x_0\in \Fp$. Since $p \nmid \disc(f)$ we know that the root $x_0$ is a singular root of $f(x)$ modulo $p$, i.e.\ $f'(x_0)\not\equiv 0 \Mod{p}$. This allows us now to use Hensel's lemma to lift the root $x_0$ of $f(x)$ modulo $p$ to the root $x\in \Qp$ of $f(x)$ which gives us that the point $(x,0)\in H_q(\Qp)$.

    \item Assume that $p\mid \disc(f)$. Depending on which conditions we have for $q$ by using the isomorphism from Lemma~\ref{lemma_iso} we get that $H_q(\Qp)\neq \emptyset$ since it will be isomorphic to some of the curves $H_1$, $H_u$, $H_p$ or $H_{pu}$ (if $p>2$) or to $H_1$, $H_{-1}$, $H_{5}$, $H_{-5}$, $H_{2}$, $H_{-2}$, $H_{10}$ or $H_{-10}$ (if $p=2$).  
\end{enumerate}

\end{proof}

\section{Filtration method for primes of bad reduction and proof of Theorem~\ref{thm_F(s)_formula}}\label{sec:3}

We want to express $F(s)$ in terms of $g(s)$ and its twist by certain Dirichlet characters. In order to do so we need to "filter out" all the positive square-free integers appearing in $g(s)$ that satisfy certain congruent conditions modulo $8$ and modulo prime factors of $\disc(f)$. 

Let $\chi_1$, $\chi_2$, $\chi_3$ and $\chi_4$ be all Dirichlet characters modulo $8$ defined as in the table:
\begin{center}
    \begin{tabular}{| c | c | c | c | c | }
    \hline
    $n$ & $1$ & $3$ & $5$ & $7$ \\
    \hline
    $\chi_1(n)$ & $1$ & $1$ & $1$ & $1$ \\
    \hline
    $\chi_2(n)$ & $1$ & $1$ & $-1$ & $-1$ \\
    \hline
    $\chi_3(n)$ & $1$ & $-1$ & $-1$ & $1$ \\
    \hline
    $\chi_4(n)$ & $1$ & $-1$ & $1$ & $-1$ \\
    \hline
    \end{tabular}
\end{center}

Let $g(s) = \sideset{}{'}\prod (1+p^{-s})= \sideset{}{'}\sum q^{-s}$ be defined as before where $\sideset{}{'}\sum $ denotes the sum over all square-free integers whose prime factors have factorization type $(1,1,1,1)$, $(1,1,2)$ or $(1,3)$. Let's say that we want to filter all square-free integers $q$ from that sum which are congruent to $1$ modulo $8$. We can do that by calculating the following sum :
\[
    g(s)+g^{\chi_2}(s)+g^{\chi_3}(s)+g^{\chi_4}(s) = \sideset{}{'}\sum \frac{\chi_1(q)+\chi_2(q)+\chi_3(q)+\chi_4(q)}{q^s}=4\sideset{}{'}\sum_{q\equiv 1 \Mod{8}} \frac{1}{q^s}.
\]
We can do similar things to filter out the $q$'s that are congruent to $3$, $5$ or $7$ modulo $8$ and get the following lemma:
\begin{lemma}[filtration modulo $8$] \label{filter_mod8}
    With the notation as above the following holds:
    \begin{enumerate}[i)]
        \item \[ \sideset{}{'}\sum_{q\equiv 1 \Mod{8}} \frac{1}{q^s} = \frac{1}{4}(g(s)+g^{\chi_2}(s)+g^{\chi_3}(s)+g^{\chi_4}(s)). \]
        
        \item \[ \sideset{}{'}\sum_{q\equiv 3 \Mod{8}} \frac{1}{q^s} = \frac{1}{4}(g(s)+g^{\chi_2}(s)-g^{\chi_3}(s)-g^{\chi_4}(s)). \]
        
        \item \[ \sideset{}{'}\sum_{q\equiv 5 \Mod{8}} \frac{1}{q^s} = \frac{1}{4}(g(s)-g^{\chi_2}(s)-g^{\chi_3}(s)+g^{\chi_4}(s)). \]
    
        \item \[ \sideset{}{'}\sum_{q\equiv 7 \Mod{8}} \frac{1}{q^s} = \frac{1}{4}(g(s)-g^{\chi_2}(s)+g^{\chi_3}(s)-g^{\chi_4}(s)). \]
    \end{enumerate}
\end{lemma}

Lets denote by $\psi_r$ the quadratic Dirichlet character modulo $r$, i.e.\ $\psi_r (n) = \legendre{n}{r}$. By considering similar sums as for filtration modulo $8$ we get:

\begin{lemma}[filtration of squares modulo $r$] \label{filter_mod_r}
    With the notation as above the following holds:
    \begin{enumerate}[i)]
        \item \[ \sideset{}{'}\sum_{\legendre{q}{r}=1} \frac{1}{q^s} = \frac{1}{2}(g(s)+g^{\psi_r} (s)). \]

        \item \[ \sideset{}{'}\sum_{\legendre{q}{r}=-1} \frac{1}{q^s} = \frac{1}{2}(g(s)-g^{\psi_r} (s)). \]
    \end{enumerate}
\end{lemma}

\begin{theorem}\label{thm_F(s)_formula}
    Let $f(x) \in \Z[x]$ be a irreducible and monic polynomial of degree $4$, $q$ a square-free integer and ${H_q \, :\, qy^2=f(x)}$ a curve. Let $F(s)$ and $g(s)$ be the generating and auxiliary Dirichlet series defined as above. Let $r_1$, $r_2$, \ldots, $r_l >2$ be all prime factors that divide $\operatorname{disc}(f)$. Furthermore, let $C_1$, $C_2$, \ldots, $C_n$ be sets of positive square-free integers satisfying all the different modular conditions modulo $r_j$ and $8$ obtained using Proposition~\ref{els_crit} (see Example~\ref{example_1} for a detailed description on how this sets look like). The following holds
    \begin{equation}\label{eq:F(s)_formula}
        F(s)= \frac{1}{2^{l+2}}\sum_{r=1}^n \sum_{\substack{1 \leq i \leq 4\\ 0 \leq k \leq l \\ 1\leq j_1<\dotso<j_k\leq l}} \varepsilon_{(r,i,j_1, \dotsc, j_k)}\cdot g^{\chi_i \psi_{r_{j_1}}\dotsm \psi_{r_{j_k}}} (s)\cdot \left(2^{e_{(r,i,j_1, \dotsc, j_k)}}r_1^{e^1_{(r,i,j_1, \dotsc, j_k)}}\dotsm r_l^{e^l_{(r,i,j_1, \dotsc, j_k)}}\right)^{-s}
    \end{equation}
    where $\varepsilon_{(r,i,j_1, \dotsc, j_k)} \in \{ \pm 1\}$ and $e_{(r,i,j_1, \dotsc, j_k)},\ e^m_{(r,i,j_1, \dotsc, j_k)} \in \{0, 1\}$ are chosen depending on the set $C_r$.
\end{theorem}

\begin{proof}[Proof of Theorem~\ref{thm_F(s)_formula}]
    By using Proposition~\ref{els_crit} and applying Lemma~\ref{filter_mod8} and~\ref{filter_mod_r} on the Dirichlet series $g(s)$ multiple times we can filter out  all the different conditions that $q$ has to satisfy modulo $r_j$ and modulo $8$. Multiplying the new Dirichlet series that we got by filtration with the factor $\left(2^{e_{(r,i,j_1, \dotsc, j_k)}}r_1^{e^1_{(r,i,j_1, \dotsc, j_k)}}\dotsm r_l^{e^l_{(r,i,j_1, \dotsc, j_k)}}\right)^{-s}$ for a suitable choice of $e_{(r,i,j_1, \dotsc, j_k)}$, $e^1_{(r,i,j_1, \dotsc, j_k)}$, \ldots, $e^l_{(r,i,j_1, \dotsc, j_k)} \in \{0,1\}$ depending on the conditions for $q$ modulo $r_j$ and modulo $8$. Finally, by adding them all up we get the desired formula for $F(s)$.
\end{proof}

\section{Some results about frobenian multiplicative functions}\label{sec:frobenian_functions}
We will start by recalling some useful facts about frobenian multiplicative functions following~\cite{loughran2022frobenian}. 

\begin{definition}\label{def_2.1.}
    Let $\rho : \operatorname{Val}(\Q) \to \C$ be a function. We say that $\rho$ is \textbf{frobenian} if there exist
    \begin{enumerate}[(a)]
        \item A finite Galois extension $K/\Q$, with Galois group $\Gamma$,
        \item finite set of primes $S$ containing all the primes ramifying in $K$,
        \item A class function $\varphi : \Gamma \to \C$,
    \end{enumerate}
    such that for all primes $p \notin S$ we have \[ \rho(p)=\varphi(\operatorname{Frob}_p), \] where $\operatorname{Frob}_p \in \Gamma$ is the Frobenius element of $p$. We define the \textbf{mean} of $\rho$ to be \[ m(\rho) = \frac{1}{\abs{\Gamma}}\sum_{\gamma \in \Gamma} \varphi(\gamma). \]
\end{definition}

\begin{lemma}\label{lemma_2.3.}
    Let $\rho_1$ and $\rho_2$ be frobenian functions. Then $\rho_1 \cdot \rho_2$ is also frobenian.
\end{lemma}

\begin{lemma}\label{lemma_2.4.}
    Let $\rho$ be a frobenian function with $m(\rho)\neq 0$. The following holds as $x \to \infty$:
    \begin{enumerate}[(1)]
        \item 
        \[ 
            \sum_{p\leq x} \rho(p) = m(\rho)\cdot \operatorname{Li}(x) + \bigO{x\exp{(-c\sqrt{\ln{x}})}}, \text{ for some constant } c>0.
        \]

        \item 
        \[
            \sum_{p \leq x} \frac{\rho(p)}{p}=m(\rho) \ln{\ln{x}}+ C_\rho + \bigO{\frac{1}{\ln{x}}}, \text{ for some constant } C_\rho.
        \]

        \item 
        \[
           \sum_{p \leq x} \rho(p)\ln{p} = m(\rho)\cdot x +\bigO{x\exp{(-c\sqrt{\ln{x}})}}, \text{ for some constant } c>0.
        \]

        \item 
        \[
            \prod_{\substack{p \leq x \\ \abs{\rho(p)}<p}}\left(1+\frac{\rho(p)}{p}\right) \sim C'_\rho (\ln{x})^{m(\rho)},
        \]
        for some constant $C'_\rho \neq 0$, where $C'_\rho$ is real and positive when $\rho$ is real-valued.
    \end{enumerate}
\end{lemma}

\begin{lemma}\label{lemma_2.5.}
    Let $\rho$ be a frobenian function.
    \begin{enumerate}[(1)]
        \item Only finitely many primitive Dirichlet characters $\chi$ satisfy $m(\rho\chi)\neq 0$.
    \end{enumerate}
    Assume that $\rho$ is real valued and non-negative and let $\chi$ be a Dirichlet character.
    \begin{enumerate}
        \item[(2)] We have $\abs{m(\rho\chi)}\leq m(\rho)$.
        \item[(3)] The following are equivalent:
        \begin{enumerate}[(a)]
            \item $\abs{m(\rho\chi)}=m(\rho)$,
            \item $m(\rho\chi)=m(\rho)$,
            \item $(\rho\chi)(p)=\rho(p)$ for all but finitely many primes $p$.
        \end{enumerate}
    \end{enumerate}
\end{lemma}

\begin{definition}\label{def_2.6.}
    Let $\varepsilon \in \langle0,1\rangle$ and let $\rho : \N \to \C$ be a multiplicative function. We say that $\rho$ is an \textbf{$\varepsilon$-weak frobenian multiplicative function} if 
    \begin{enumerate}[(1)]
        \item The restriction of $\rho$ to the set of primes is a frobenian function, in the sense of Definition~\ref{def_2.1.},
        \item $\abs{\rho(n)}=\bigO{n^\varepsilon}$ for all $n \in \N$.
        \item There exists a constant $H\in \N$ such that $\abs{\rho(p^k)} \leq H^k$ for all primes $p$ and all $k\geq 1$.
    \end{enumerate}
    We define the \textit{mean} of $\rho$ to be the mean of the corresponding frobenian function.
\end{definition}

\begin{definition}\label{def_2.7.}
    Let $\rho : \N \to \C$ be a multiplicative function. We say that $\rho$ is a \textbf{frobenian multiplicative function} if it is $\varepsilon$-weak frobenian for all $\varepsilon \in \langle0,1\rangle$.
\end{definition}

\begin{remark}\label{rmk_prod_of_frobenians}
    Note that if $\rho_1$ and $\rho_2$ are ($\varepsilon$-weak) frobenian multiplicative functions, then by Lemma~\ref{lemma_2.3.} and Definition~\ref{def_2.6.} $\rho_1 \rho_2$ is also ($\varepsilon$-weak) frobenian multiplicative functions. In particular $\rho\chi$ is a frobenian multiplicative function for a Dirichlet character $\chi$ and a frobenian multiplicative function $\rho$.
    \end{remark}

\begin{lemma}\label{lemma_2.8.}
    Let $\varepsilon \in \langle0,1\rangle$ and $\rho$ be an $\varepsilon$-weak frobenian multiplicative function. Then,
    \[
        \sum_{n \leq x} \rho(n)=c_\rho x(\ln{x})^{m(\rho)-1}+ \bigO{x(\ln{x})^{m(\rho)-2}},
    \]
    where 
    \[
        c_\rho = \frac{1}{\Gamma(m(\rho))}\prod_p \left(\sum_{k=0}^\infty \frac{\rho(p^k)}{p^k}\right)\left(1-\frac{1}{p}\right)^{m(\rho)}. 
    \]
    If $\rho$ is real-valued and non-negative with $m(\rho) \neq 0$, then $c_\rho$ is real and positive.
\end{lemma}

\section[Calculating the asymptotic behaviour of $L(x)$]{Calculating the asymptotic behaviour of L(x)}\label{sec:5}

We can finally calculate the asymptotic behaviour of $L(x)$ using Theorem~\ref{thm_F(s)_formula} and Lemma~\ref{lemma_2.8.}.

Let $G$ be the Galois group of the polynomial $f(x)$. Denote by $A$, $B$ and $C$ the conjugacy classes in the Galois group $G$ that correspond to the factorization types $(1,1,1,1)$, $(1,1,2)$ and $(1,3)$, i.e.\ the set $A$ is the set of all $\sigma \in G$ that have the cyclic type $(1,1,1,1)$ (that is $A=\{\operatorname{id}\}$), $B$ is the set of all $\sigma \in G$ that have the cyclic type $(1,1,2)$ and $C$ is the set of all $\sigma \in G$ that have the cyclic type $(1,3)$. 

We define a class function $\varphi :G \to \C$ that is the indicator function for the classes $A$, $B$ and $C$, $\varphi = 1_{A}+ 1_{B}+ 1_{C}$. Now we can define the frobenian function $\rho(p) = \varphi(\operatorname{Frob}_p)$. By setting $\rho(p^k)=0$ for all primes $p$ and all $k\geq 2$ and extending $\rho$ to $\N$ by multiplicativity we get that $\rho : \N \to \C$ is in fact a multiplicative frobenian function that corresponds to the Dirichlet series $g(s)$. Furthermore, by Remark~\ref{rmk_prod_of_frobenians} we also have that $\rho\chi_i \psi_{r_{j_1}} \dotsm \psi_{r_{j_k}}$ is a multiplicative frobenian function that corresponds to the Dirichlet series $g^{\chi_i \psi_{r_{j_1}} \dotsm \psi_{r_{j_k}}}(s)$. 

Let $F(s) = \sum_{n=1}^\infty \frac{a_n}{n^s}$ be a Dirichlet series as in Theorem~\ref{thm_F(s)_formula}. By definitions of $F(s)$ and $L(x)$ we have that $L(x) = \sum_{n \leq x} a_n$.

Using the equation~\eqref{eq:F(s)_formula} we get
\begin{equation}\label{eq:L(x)_unfinished}
    L(x) = \frac{1}{2^{l+2}}\sum_{r=1}^n\sum_{\substack{1 \leq i \leq 4\\ 0 \leq k \leq l \\ 1\leq j_1<\dotso<j_k\leq l}} \varepsilon_{(r,i,j_1, \dotsc, j_k)}\sum_{n\leq x_{(r,i,j_1, \dotsc, j_k)}} (\rho\chi_i \psi_{r_{j_1}} \dotsm \psi_{r_{j_k}})(n)
\end{equation}
where $x_{(r,i,j_1, \dotsc, j_k)} = x\cdot 2^{-e_{(r,i,j_1, \dotsc, j_k)}}r_1^{-e^1_{(r,i,j_1, \dotsc, j_k)}}\dotsm r_l^{-e^l_{(r,i,j_1, \dotsc, j_k)}}$.

From Lemma~\ref{lemma_2.8.} it follows that
\begin{align*}
    &\sum_{n\leq x_{(r,i,j_1, \dotsc, j_k)}} (\rho\chi_i \psi_{r_{j_1}} \dotsm \psi_{r_{j_k}})(n) = \\
    &= \frac{c_{\rho\chi_i \psi_{r_{j_1}} \dotsm \psi_{r_{j_k}}}}{2^{e_{(r,i,j_1, \dotsc, j_k)}}r_1^{e^1_{(r,i,j_1, \dotsc, j_k)}}\dotsm r_l^{e^l_{(r,i,j_1, \dotsc, j_k)}}}\cdot \frac{x}{(\ln{x})^{1-m(\rho\chi_i \psi_{r_{j_1}} \dotsm \psi_{r_{j_k}})}} +\bigO{\frac{x}{(\ln{x})^{2-m(\rho\chi_i \psi_{r_{j_1}} \dotsm \psi_{r_{j_k}})}}}\\
\end{align*}

In the last equation we used that 
\[
    \frac{x}{\left(\ln{x}-\ln{\left({2^{e_{(r,i,j_1, \dotsc, j_k)}}r_1^{e^1_{(r,i,j_1, \dotsc, j_k)}}\dotsm r_l^{e^l_{(r,i,j_1, \dotsc, j_k)}}}\right)}\right)^{2-m(\rho\chi_i \psi_{r_{j_1}} \dotsm \psi_{r_{j_k}})}} = \bigO{\frac{x}{(\ln{x})^{2-m(\rho\chi_i \psi_{r_{j_1}} \dotsm \psi_{r_{j_k}})}}}
\]
and
\begin{align*}
        &\frac{x}{\left(\ln{x}-\ln{\left({2^{e_{(r,i,j_1, \dotsc, j_k)}}r_1^{e^1_{(r,i,j_1, \dotsc, j_k)}}\dotsm r_l^{e^l_{(r,i,j_1, \dotsc, j_k)}}}\right)}\right)^{1-m(\rho\chi_i \psi_{r_{j_1}} \dotsm \psi_{r_{j_k}})}} - \frac{x}{(\ln{x})^{1-m(\rho\chi_i \psi_{r_{j_1}} \dotsm \psi_{r_{j_k}})}} = \\
        &=\bigO{\frac{x}{(\ln{x})^{2-m(\rho\chi_i \psi_{r_{j_1}} \dotsm \psi_{r_{j_k}})}}}  
\end{align*}
since $\lim_{x\to \infty}\frac{(\ln{x}-c)^{m-1}-(\ln{x})^{m-1}}{(\ln{x})^{m-2}}=c(1-m)$, where $c>0$ and $m<1$ are constants.

By Lemma~\ref{lemma_2.5.} we have that $\abs{m(\rho\chi_i \psi_{r_{j_1}} \dotsm \psi_{r_{j_k}})}\leq m(\rho)$ which means that the term $\frac{x}{(\ln{x})^{1-m(\rho)}}$ is going to be the main term in $L(x)$.

Now by plugging all those equations back into \eqref{eq:L(x)_unfinished} we get the following result.

\begin{theorem}\label{L(x)_formula}
    With the notation as above, we have that
    \[
        L(x) = c_f \frac{x}{(\ln{x})^{1-m(\rho)}}+\bigO{\frac{x}{(\ln{x})^\alpha}}
    \]
    where 
    \[
        \alpha = \min\{2-m(\rho),1-m(\rho\chi_i \psi_{r_{j_1}} \dotsm \psi_{r_{j_k}}) : m(\rho\chi_i \psi_{r_{j_1}} \dotsm \psi_{r_{j_k}})\neq m(\rho)\},
    \]
    \[
        c_f= \frac{1}{2^{l+2}}\sum_{r=1}^n\sum_{\substack{1 \leq i\leq 4\\0 \leq k \leq l\\1\leq j_1<\dotso < j_k \leq l\\m(\rho\chi_i \psi_{r_{j_1}} \dotsm \psi_{r_{j_k}})=m(\rho)}}\varepsilon_{(r,i,j_1, \dotsc, j_k)} \cdot \frac{c_{\rho\chi_i \psi_{r_{j_1}} \dotsm \psi_{r_{j_k}}}}{2^{e_{(r,i,j_1, \dotsc, j_k)}}r_1^{e^1_{(r,i,j_1, \dotsc, j_k)}}\dotsm r_l^{e^l_{(r,i,j_1, \dotsc, j_k)}}},
    \]
    \[
        c_{\rho\chi_i \psi_{r_{j_1}} \dotsm \psi_{r_{j_k}}} = \frac{1}{\Gamma(m(\rho\chi_i \psi_{r_{j_1}} \dotsm \psi_{r_{j_k}}))}\prod_p \left(\sum_{k=0}^\infty \frac{\rho\chi_i \psi_{r_{j_1}} \dotsm \psi_{r_{j_k}}(p^k)}{p^k}\right)\left(1-\frac{1}{p}\right)^{m(\rho\chi_i \psi_{r_{j_1}} \dotsm \psi_{r_{j_k}})}
    \]
    and
    \[
    m(\rho) = 
    \begin{cases}
        1/4, & \text{if } G \cong \left(\Z/2\Z\right) \times \left(\Z/2\Z\right) \text{ or } G \cong \Z/4\Z \\
        3/8, & \text{if } G \cong D_4 \\
        3/4, & \text{if } G \cong A_4 \\
        5/8, & \text{if } G \cong S_4\text{.}\\
    \end{cases}
    \]
\end{theorem}

Finally, we are now able to prove Corollary~\ref{L(x)_formula_shortv}. 

\begin{proof}[Proof of Corollary~\ref{L(x)_formula_shortv}]
    Corollary~\ref{L(x)_formula_shortv} follows almost directly form Theorem~\ref{L(x)_formula} (by setting $m=1-m(\rho)$). It remains to prove that $m<\alpha \leq 1+m$ and $c_f>0$.

    The inequality $m<\alpha \leq 1+m$ follows directly from the definition of $\alpha$ in Theorem~\ref{L(x)_formula}.

    To show that $c_f>0$ it is sufficient to prove that there exists a positive lower bound for the constant $c_f$.

    Let $n=8\cdot3\cdot5\cdot\disc(f)$. In order to show this lower bound we will compute the asymptotics for the number of positive square-free $q$'s such that $q \equiv 1 \Mod{n}$ and that all prime factors $p$ of $q$ have factorization type $(1,1,1,1)$, $(1,1,2)$ or $(1,3)$. Note that for any such $q$ the curve $H_q$ is ELS. Namely, the starting curve $H$ is ELS since the polynomial $f(x)$ is monic and because of that $H$ is birationally isomorphic to an elliptic curve (since $H$ has a rational point at infinity). Now, since $q$ is a square locally at all "bad" primes, and satisfies the local solubility conditions everywhere else it follows that $H_q$ is indeed ELS.

    For a positive real number $x$, denote with $N(x)$ the number of such $q$'s in the range $0<q\leq x$.

    Let $a$ be a frobenian multiplicative function defined by
    \[
        a(p)=
        \begin{cases}
            0, & \text{if } p \mid n, \\
            0, & \text{if } p \text{ doesn't have factorization type } (1,1,1,1),\, (1,1,2),\, (1,3), \\
            1, & \text{otherwise } \\
        \end{cases}
    \]
    and $a(p^k)=0$ for all $k\geq 2$ and all primes $p$.
    Since $a$ is non-negative, real valued, and has non-zero mean $m(a)=1-m$, Lemma~\ref{lemma_2.8.} gives
    \[
        \sum_{q\leq x}a(q)=c\frac{x}{(\ln{x})^m}+\bigO{\frac{x}{(\ln{x})^{m+1}}},
    \]
    for some real constant $c>0$.

    Next, we "filter out" the $q$'s that satisfy the congruence condition $q \equiv 1 \Mod{n}$. Let $X(n)$ denote the group of all Dirichlet characters modulo $n$ and let $\varphi$ denote the Euler's totient function. Then we can write the indicator function for numbers congruent to $1$ modulo $n$ as
    \[
        1_{\{q \equiv 1 \Mod{n}\}}=\frac{1}{\varphi(n)}\sum_{\chi \in X(n)}\chi.
    \]
    Thus we have that
    \begin{equation}\label{eq:N(x)}
        N(x)=\frac{1}{\varphi(n)}\sum_{\chi \in X(n)}\sum_{q\leq x}a(q)\chi(q).       
    \end{equation}

    For a Dirichlet character $\chi \in X(n)$, we consider now the frobenian multiplicative function $q \mapsto a(q)\chi(q)$. We have the following two cases:
    \begin{enumerate}
        \item[$1^\circ$] $a(p)\chi(p)=a(p)$ holds for all primes $p$ coprime to $n$. \newline 
        By multiplicativity we can conclude that $a(q)\chi(q)=a(q)$ for all positive integers $q$. Hence
        \[
            \sum_{q\leq x} a(q)\chi(q)=\sum_{q\leq x} a(q)= c\frac{x}{(\ln{x})^m}+\bigO{\frac{x}{(\ln{x})^{m+1}}}
        \]
        where the constant $c>0$ is as above.

        \item[$2^\circ$]  $a(p)\chi(p) \neq a(p)$ for some prime $p$ coprime to $n$. \newline
        Let $L=K(\zeta_n)$ where $K$ is the splitting filed of the polynomial $f(x)$ and $\zeta_n$ a primitive root of unity modulo $n$. Observe that both $a(p)$ and $\chi(p)$ are determined by the Frobenius value at $p$ in $\Gal(L/\Q)$. By the Chebotarev density theorem, there are infinitely many primes $l$ coprime to $n$ such that $\Frob_l \in \Gal(L/\Q)$ is equal to $\Frob_p$ up to conjugacy. For such primes $l$ we also have that $a(l)\chi(l) \neq a(l)$. Thus there are infinitely many primes $l$ such that $a(l)\chi(l) \neq a(l)$. By Lemma~\ref{lemma_2.5.} this forces $\abs{m(a\chi)}<m(a)$ and thus by Lemma~\ref{lemma_2.8.} we have
        \[
            \sum_{q\leq x} a(q)\chi(q)= \smallO{\frac{x}{(\ln{x})^m}}.
        \]
    \end{enumerate}
    By plugging all this back in \eqref{eq:N(x)} we get
    \[
        N(x)=rc\frac{x}{(\ln{x})^m}+\bigO{\frac{x}{(\ln{x})^\beta}}
    \]
    for some constant $\beta >m$, where $r$ is the number of Dirichlet characters $\chi \in X(n)$ such that $a(p)\chi(p)=a(p)$ for all primes $p$ coprime to $n$. Since the trivial character is one such character, we have that $r\geq 1$. Thus, we have that $rc>0$ is the required lower bound for $c_f$.
    
\end{proof}

\section[$F(s)$ via Dedekind zeta functions]{F(s) via Dedekind zeta functions}\label{sec:6}

We want to express $F(s)$ in terms of Dedekind zeta functions of certain number fields. 

We first start by proving Proposition~\ref{thm_g(s)_via_dedekind} in which we express $g(s)$ in terms of Dedekind zeta functions of certain number fields.

\begin{proof}[Proof of Proposition~\ref{thm_g(s)_via_dedekind}]
    \begin{enumerate}[i)]
        
        \item Since $[L:\Q]=\abs{G}=4$ we have that $L=K$ is a Galois extension of $\Q$. By analysing all the different factorization types of primes $p$ in the extension $\Q \subset K$ we can rewrite the Euler product of $\zeta_K(s)$ as
        \[
            \zeta_K(s)=\prod_{p \text{ of type } (1,1,1,1) }\left(1-\frac{1}{p^s}\right)^{-4} \cdot \prod_{p \text{ of type } (2,2) }\left(1-\frac{1}{p^{2s}}\right)^{-2} \cdot \prod_{p \text{ of type } (4) }\left(1-\frac{1}{p^{4s}}\right)^{-1} \cdot \prod_{p \text{ ramified }}(\dotso).
        \]
        From here we finally get that 
        \[
            g(s)^4\cdot F_K(s)=\frac{\zeta_K(s)}{\zeta_K(2s)}
        \]
        where 
        \[
            F_K(s)=\prod_{p \text{ of type } (2,2) }\left(1+\frac{1}{p^{2s}}\right)^{2} \cdot \prod_{p \text{ of type } (4) }\left(1+\frac{1}{p^{4s}}\right) \cdot \prod_{p \text{ ramified }}(\dotso)
        \] is a \textit{nice} function.

        \item Similarly as in i) we start by rewriting the Euler products of $\zeta_K(s)/\zeta_K(2s)$ and $\zeta_L(s)/\zeta_L(2s)$ (notice that a prime $p$  can't have factorization type $(1,3)$ in $L$ since there is no conjugacy class in $D_4$ that correspond to the type $(1,3)$).
        \[
            \frac{\zeta_K(s)}{\zeta_K(2s)}=\prod_{p \text{ of type } (1,1,1,1)}\left(1+\frac{1}{p^s}\right)^{4}\cdot \prod_{p \text{ of type } (1,1,2)}\left[\left(1+\frac{1}{p^{2s}}\right)\left(1+\frac{1}{p^s}\right)^2\right]\cdot (\text{nice function})
        \]
        \[
            \frac{\zeta_L(s)}{\zeta_L(2s)}=\prod_{p \text{ of type } (1,1,1,1)}\left(1+\frac{1}{p^s}\right)^{8}\cdot \prod_{p \text{ of type } (1,1,2)}\left(1+\frac{1}{p^{2s}}\right)^4 \cdot (\text{nice function}).
        \]
        From here we finally get
        \[
            g(s)^8 \sim \left( \frac{\zeta_K (s)}{\zeta_K (2s)} \right)^4 \cdot \left( \frac{\zeta_L (s)}{\zeta_L (2s)} \right)^{-1}.
        \]

        \item Similarly as in ii) we rewrite the Euler products of $\zeta_K(s)/\zeta_K(2s)$ and $\zeta_L(s)/\zeta_L(2s)$
        \[
            \frac{\zeta_K(s)}{\zeta_K(2s)}=\prod_{p \text{ of type } (1,1,1,1)}\left(1+\frac{1}{p^s}\right)^{4}\cdot \prod_{p \text{ of type } (1,3)}\left[\left(1+\frac{1}{p^{3s}}\right)\left(1+\frac{1}{p^s}\right)\right]\cdot (\text{nice function})
        \]
        \[
            \frac{\zeta_L(s)}{\zeta_L(2s)}=\prod_{p \text{ of type } (1,1,1,1)}\left(1+\frac{1}{p^s}\right)^{12}\cdot \prod_{p \text{ of type } (1,3)}\left(1+\frac{1}{p^{3s}}\right)^4 \cdot (\text{nice function})
        \]
        from where we get that
        \[
            g(s)^4 \sim \left( \frac{\zeta_K (s)}{\zeta_K (2s)} \right)^4 \cdot \left( \frac{\zeta_L (s)}{\zeta_L (2s)} \right)^{-1}.
        \]

        \item We start by choosing $\Z/3\Z \cong \langle(123)\rangle \leq S_4$ a Sylow 3-subgroup of $S_4$ and denote by $L^{\Z/3\Z}$ the fixed field of the subgroup $\langle(123)\rangle$. By analysing all the different factorization types of primes $p$ in the extensions $\Q \subset K$, $\Q \subset L$ and $\Q \subset L^{\Z/3\Z}$ we rewrite the Euler product of $\zeta_K(s)/\zeta_K(2s)$, $\zeta_K(s)/\zeta_K(2s)$ and $\zeta_{L^{\Z/3\Z}}(s)/\zeta_{L^{\Z/3\Z}}(2s)$
        \begin{align}\label{eq_euler_prod_zK}
            \frac{\zeta_K(s)}{\zeta_K(2s)} = 
            &\prod_{p \text{ of type } (1,1,1,1)}\left(1+\frac{1}{p^s}\right)^{4}\cdot \prod_{p \text{ of type } (1,1,2)}\left[\left(1+\frac{1}{p^{2s}}\right)\left(1+\frac{1}{p^s}\right)^2\right] \\ \nonumber
            & \cdot \prod_{p \text{ of type } (1,3)}\left[\left(1+\frac{1}{p^{3s}}\right)\left(1+\frac{1}{p^s}\right)\right]\cdot (\text{nice function}),
        \end{align}
        \begin{align}\label{eq_euler_prod_zL}
            \frac{\zeta_L(s)}{\zeta_L(2s)}=
            &\prod_{p \text{ of type } (1,1,1,1)}\left(1+\frac{1}{p^s}\right)^{24}\cdot \prod_{p \text{ of type } (1,1,2)}\left(1+\frac{1}{p^{2s}}\right)^{12}\\ \nonumber
            &\cdot \prod_{p \text{ of type } (1,3)}\left(1+\frac{1}{p^{3s}}\right)^{8}\cdot (\text{nice function}),
        \end{align}
        \begin{align}\label{eq_euler_prod_zL3}
            \frac{\zeta_{L^{\Z/3\Z}}(s)}{\zeta_{L^{\Z/3\Z}}(2s)}=
            &\prod_{p \text{ of type } (1,1,1,1)}\left(1+\frac{1}{p^s}\right)^{8}\cdot \prod_{p \text{ of type } (1,1,2)}\left(1+\frac{1}{p^{2s}}\right)^{4} \\ \nonumber
            &\cdot \prod_{p \text{ of type } (1,3)}\left[\left(1+\frac{1}{p^{s}}\right)\left(1+\frac{1}{p^{3s}}\right)\right]^{2}\cdot (\text{nice function}).
        \end{align}
    From \eqref{eq_euler_prod_zL} we now get that 
    \[
        \prod_{p \text{ of type } (1,1,1,1)}\left(1+\frac{1}{p^s}\right) \sim \left(\frac{\zeta_L(s)}{\zeta_L(2s)}\right)^{\frac{1}{24}}.
    \] 
    Plugging this into \eqref{eq_euler_prod_zL3} gives us
    \[
        \prod_{p \text{ of type } (1,3)}\left(1+\frac{1}{p^s}\right) \sim \left(\frac{\zeta_{L^{\Z/3\Z}}(s)}{\zeta_{L^{\Z/3\Z}}(2s)}\right)^{\frac{1}{2}}\cdot \prod_{p \text{ of type } (1,1,1,1)}\left(1+\frac{1}{p^s}\right)^{-4} \sim \left(\frac{\zeta_{L^{\Z/3\Z}}(s)}{\zeta_{L^{\Z/3\Z}}(2s)}\right)^{\frac{1}{2}} \cdot \left(\frac{\zeta_L(s)}{\zeta_L(2s)}\right)^{-\frac{1}{3}}.
    \]  
    Finally by plugging all this into \eqref{eq_euler_prod_zK} yields to
    \begin{align*}
        \prod_{p \text{ of type } (1,1,2)}\left(1+\frac{1}{p^s}\right) 
        &\sim \left(\frac{\zeta_K(s)}{\zeta_K(2s)}\right)^{\frac{1}{2}} \cdot \prod_{p \text{ of type } (1,1,1,1)}\left(1+\frac{1}{p^s}\right)^{-2} \cdot \prod_{p \text{ of type } (1,3)}\left(1+\frac{1}{p^s}\right)^{-\frac{1}{2}} \\*
        &\sim \left(\frac{\zeta_K(s)}{\zeta_K(2s)}\right)^{\frac{1}{2}} \cdot \left(\frac{\zeta_L(s)}{\zeta_L(2s)}\right)^{\frac{1}{12}} \cdot \left(\frac{\zeta_{L^{\Z/3\Z}}(s)}{\zeta_{L^{\Z/3\Z}}(2s)}\right)^{-\frac{1}{4}}.
    \end{align*}

    Combining the previous equations together gives us
    \[
            g(s)^{24} \sim \left( \frac{\zeta_K (s)}{\zeta_K (2s)} \right)^{12} \cdot \left( \frac{\zeta_L (s)}{\zeta_L (2s)} \right)^{-5} \cdot \left( \frac{\zeta_{L^{\Z/3\Z}} (s)}{\zeta_{L^{\Z/3\Z}} (2s)} \right)^6.
    \]
    
    \end{enumerate}
\end{proof}

In order to express $F(s)$ via Dedekind zeta functions of certain number fields we will need the following Lemma about Dirichlet series and twists with Dirichlet characters.

\begin{lemma}
    For Dirichlet series $F(s)$ and $G(s)$ and Dirichlet characters $\chi$ and $\psi$ the following holds :
    \begin{enumerate}[i)]
        \item $(F+G)^\chi=F^\chi + G^\chi$. 
        \item $\left(F^\chi\right)^\psi=F^{\chi \cdot \psi}$.
        \item $(F\cdot G)^\chi=F^\chi \cdot G^\chi$.
        \item $\left(\frac{1}{F}\right)^\chi = \frac{1}{F^\chi}$ for $F \neq 0$.
    \end{enumerate}
\end{lemma}
\begin{proof}
    Let $F(s)=\sum_{n=1}^{\infty} \frac{a_n}{n^s}$ and $G(s)=\sum_{n=1}^{\infty} \frac{b_n}{n^s}$. We have the following :
    \begin{enumerate}[i)]
        \item
        \[
            (F+G)^\chi (s) = \sum_{n=1}^\infty \frac{(a_n+b_n)\chi(n)}{n^s} = \sum_{n=1}^\infty \frac{a_n \chi(n)}{n^s} + \sum_{n=1}^\infty \frac{b_n \chi(n)}{n^s} = F^\chi (s) + G^\chi (s).
        \]
        \item 
        \[
            \left(F^\chi\right)^\psi (s) = \sum_{n=1}^\infty \frac{a_n \chi(n) \psi(n)}{n^s} = \sum_{n=1}^\infty \frac{a_n (\chi \cdot \psi)(n)}{n^s} = F^{\chi \cdot \psi} (s).
        \]
        \item Let $(F\cdot G)(s)=\sum_{n=1}^{\infty} \frac{c_n}{n^s}$ where $c_n = \sum_{k \mid n} a_k b_{\frac{n}{k}}$. Then $(F \cdot G)^\chi (s) = \sum_{n=1}^{\infty} \frac{c_n \chi(n)}{n^s}$ and $(F^\chi \cdot G^\chi)(s)= \sum_{n=1}^{\infty} \frac{d_n}{n^s}$ where  $d_n = \sum_{k \mid n} a_k \chi(k) \cdot b_{\frac{n}{k}}\chi\left(\frac{n}{k}\right)= c_n \chi(n)$. Hence, $(F\cdot G)^\chi=F^\chi \cdot G^\chi$.
        \item Let $\left(\frac{1}{F}\right)(s)=\sum_{n=1}^{\infty} \frac{a'_n}{n^s}$ where $a'_i$'s satisfy the following conditions: $a'_1 a_1 = 1$ and $\sum_{k \mid n}a'_k a_{\frac{n}{k}}=0$ for $n \geq 2$. Now we have that $\left(\frac{1}{F}\right)^\chi(s) \cdot F\chi(s)= \sum_{n=1}^{\infty} \frac{c_n}{n^s}$ where $c_1 = a'_1\chi(1) \cdot a_1 \chi(1)=1$ and $c_n = \sum_{k \mid n} a'_k \chi(k) \cdot a_{\frac{n}{k}}\chi\left(\frac{n}{k}\right)= \chi(n)\sum_{k \mid n}a'_ka_{\frac{n}{k}}=0$ for $n \geq 2$. Hence, $\left(\frac{1}{F}\right)^\chi = \frac{1}{F^\chi}$.
    \end{enumerate}
\end{proof}

Finally, by combining Theorem~\ref{thm_F(s)_formula} and Proposition~\ref{thm_g(s)_via_dedekind} using the above Lemma we are able to express the generating Dirichlet series $F(s)$ in terms of Dedekind zeta functions of certain number fields and its twists by certain Dirichlet characters.

\section{Future work}

The paper poses a interesting question for a future project:
\begin{itemize}
    \item 
    Can we use formulas obtained in Proposition~\ref{thm_g(s)_via_dedekind} together with methods described in \cite{tenenbaum2015introduction} (Chapter~$2.5$), in \cite{serre1974divisibilite} (Chapter~$2$) or some similar methods in order to obtain better estimates for the error term in Corollary~\ref{L(x)_formula_shortv}? What about if we assume the Generalized Riemann hypothesis?    

\end{itemize}

\section*{Acknowledgements}

The author was supported by the project “Implementation of cutting-edge research and its application as part of the Scientific Center of Excellence for Quantum and Complex Systems, and Representations of Lie Algebras“, PK.1.1.02, European Union, European Regional Development Fund and by the Croatian Science Foundation under the project no. IP-2022-10-5008.

The author would like to thank Adam Morgan for pointing him to the Landau-Selberg-Delange method, for providing him with informative and interesting articles on this method and for improving the result of Corollary~\ref{L(x)_formula_shortv} by showing that the constant $c_f$ is in fact positive (where in the previous version we only had $c_f \geq 0$)  and also his mentor Matija Kazalicki on many constructive discussions about the problem.

\bibliographystyle{alpha}
\bibliography{bibliography}

\end{document}